\documentclass[a4paper,final]{siamltex}
\usepackage{amsmath,amssymb}
\usepackage{graphicx,epsfig}
\usepackage{psfrag}
\usepackage{color}
\usepackage{multirow}\usepackage{hhline}
\usepackage{amsfonts}
\usepackage{latexsym}
\usepackage{pstricks}
\usepackage{bbm}
\usepackage{caption}
\usepackage{subcaption}
\usepackage{verbatim}
\usepackage{accents}
\usepackage[color,notref,notcite]{showkeys}

\renewcommand{\theequation}{\thesection.\arabic{equation}}
\makeatletter\@addtoreset{equation}{section}\makeatother
\newtheorem{remark}[theorem]{Remark}

\newcounter{constantsnumber}

\newcommand{\CT}{\mathcal{T}}
\newcommand{\CV}{\mathcal{V}}
\newcommand{\CJ}{\mathcal{J}}

\newcommand{\M}{{M_{bc}}}
\newcommand{\Mb}{{(M_{bc})}}

\newcommand{\bn}{\mbox{\boldmath{$n$}}}

\newcommand{\BBR}{\mbox{$\mathbb{R}$}}
\newcommand{\BBN}{\mbox{$\mathbb{N}$}}
\newfont{\twelvemsb}{msbm10 at 11.6pt}

\def\snorm#1#2{|#1|_{#2}}

\usepackage[normalem]{ulem}
\usepackage{colordvi}

\usepackage[normalem]{ulem}
\definecolor{violet}{rgb}{0.580,0.,0.827}
\def\corr#1#2#3{\typeout{Warning : a correction remains in page
\thepage}%
        {\color{blue}{\ifmmode\text{\,\sout{\ensuremath{#1}}\,}\else\sout{#1}\fi}}%
        {\color{violet}{#2}}%
        {\color{red}{#3}}}

\title{A mixed finite element method for a sixth-order elliptic problem }
\author{J\'er\^ome Droniou\thanks{School of Mathematical Sciences,
Monash University, Victoria 3800, Australia. email: \texttt{jerome.droniou@monash.edu}} 
\and Muhammad Ilyas\footnotemark[2]\thanks{School of Mathematical and Physical Sciences, University of Newcastle, 
Callaghan, NSW 2308,  {\tt Bishnu.Lamichhane@newcastle.edu.au}},
\and  Bishnu P. Lamichhane\footnotemark[2]
\and 
Glen  E. Wheeler\thanks{Institute for Mathematics and its Applications, School
  Statistics, University of Wollongong, 
{\tt glenw@uow.edu.au}}
}

\begin{document}
\maketitle

\begin{abstract}
We consider a saddle-point formulation  
for a sixth-order partial differential equation and its 
finite element approximation, for two sets of boundary conditions.  
We follow the Ciarlet--Raviart formulation for the biharmonic 
problem to formulate our saddle-point problem and the 
finite element method. The new formulation allows us to use the $H^1$-conforming 
Lagrange finite element spaces to approximate the solution.  
We prove \emph{A priori} error estimates for our approach.
Numerical results are presented for linear and quadratic finite element methods. 

\end{abstract}

\begin{keywords}
sixth-order problem,  higher order partial differential equations, 
biharmonic problem, mixed finite elements, error estimates.
\end{keywords}

\begin{AMS}
65N30, 65N15, 35J35 (Primary) 35J40 (Secondary)
\end{AMS}

\pagestyle{myheadings}
\thispagestyle{plain}

\section{Introduction}

Partial differential equations (PDE) have a long and rich history of application in physical
problems.
One of their main advantages is in the modelling of ideal or desired structures \cite{you2004pde}.
In particular, one may wish to fill a curve with a solid material that satisfies certain
conditions along the boundary.
Depending on the application, there may be several constraints along the curve.
In many applications these filled curves (called \emph{components}) are fitted together to
form a larger shape.
It is natural and in some situations essential  that at least some of the derivatives of the
surface are continuous across the boundary curves.

In this context,  higher-order partial differential equations come to the fore:
for a solution of a partial differential equation of order $2k$, one may typically allow
restrictions on all derivatives up to order $(k-1)$ along the boundary curve.
This guarantees their continuity across components.

Continuity of the second derivative across boundaries, achieved by the
sixth-order PDE proposed in this article, is critical in several settings.
In the construction of automobiles, each panel is designed by a computer based on given
specifications.
Aesthetics are an important aspect, and in this regard, the composition of reflections
from the surface of a car panel must be considered.
If one prescribes only the derivatives up to first order along the boundary, then this
leaves open the possibility of the second derivative of the panel changing sign across the
boundary. In practical terms, this causes boundaries to move from being convex
to concave, or vice-versa.
Reflections will flip across such boundaries, which from an aesthetic
perspective is unacceptable.

The strength and maximal load bearing of tensile structures also depends
critically on the continuity of higher derivatives across component boundaries.
Force is optimally spread uniformly across components, however, where
derivatives of the surface are large, force and load are accumulated.
This can be by design.
It is dangerous however when force accumulates across a boundary due not to
design but to a discontinuity in one of the higher derivatives across that
boundary.
This concern can be alleviated when a number of derivatives dependent upon the
total expected load of the structure can be guaranteed to be continuous.
Two derivatives are guaranteed by our scheme and this is typically enough for
most minor structures, such as small buildings, residential homes, and
vehicles.

Sixth-order PDE have arisen in a variety of other contexts, from
propeller blade design \cite{prop} to ulcer modelling \cite{ulcer}.
Generic applications of sixth-order PDE to manufacturing are mentioned in
\cite{benson1967general,bloor1995complex}. Applications of sixth-order 
problems in surface modelling and fluid flows 
are considered in \cite{LX07,TDQ14}. 

To see that sixth-order PDE are natural for such applications, it is
instructive to view such an equation variationally.
Minimising the classical Dirichlet energy, we calculate the first variation of
the functional
\[
\int_\Omega |\nabla u|^2dx\,,
\]
and find the Laplace equation
\[
\Delta u = 0
\]
or, in the case of the gradient flow, the heat equation
\[
(\partial_t - \Delta )u = 0\,.
\]
Minimising the elastic energy, the integrand of the functional to be minimised
depends on an additional order of derivative of $u$, and so the Euler-Lagrange equation
and resulting gradient flow is of fourth-order.
If we are additionally interested in minimising the rate of change of curvature
across the surface, the `rate of change of acceleration' or \emph{jerk}, then
the functional will depend on three orders of derivatives of $u$.
The resulting Euler-Lagrange equation
\[
\Delta^3 u = 0
\]
and gradient flow
\[
(\partial_t - \Delta^3)u = 0
\]
depend on six orders of derivatives of $u$.
This perspective is taken in Section \ref{SCTformulation}, where the
variational formulation is made rigorous.
Recent resarch interest in such equations includes
\cite{rybka1,rybka2,parkins,rybka3}.

In geophysics, sixth-order PDE are used to overcome difficulties involving
complex geological faults \cite{yao2015smooth}.
Indeed, sixth-order PDE arise in a variety of geophysical contexts due to their
appearance as models in electromagneto-thermoelasticity \cite{sherief2002two}
and relation to equatorial electrojets \cite{whitehead1971equatorial}.
We remark that model PDE from geophysics are in general quite interesting to
study from a PDE perspective, with issues such as non-uniqueness and general
ill-posedness fundamental characteristics; we refer to \cite{EAA1973} for a
selection of such issues. 

The major contribution in our paper is a mixed finite element scheme for a
sixth-order partial differential equation. This allows one to accurately model
components arising from prescribed (up to and including) second order
derivatives along boundary curves.  {Another approach to approximate 
the solution of the sixth-order elliptic problem based on the interior penalty is considered by Gudi and Neilan \cite{GN11}.}
In Section \ref{SCTformulation} we introduce our setting, which considers two different
sets of boundary conditions: simply supported, and clamped.
We use constrained minimisation to cast our problems in a mixed formulation as
in the case of the biharmonic equation \cite{Cia78,DP01,Lam11a}
(other approaches to mixed formulations for the biharmonic equation can be found in
\cite{CR74,CG75,Fal78,FO80,BOP80,Mon87,Lam11b}).
The resulting saddle-point problem allows us to apply low order $H^1$-conforming finite element methods 
to approximate the solution of the sixth-order problem. This approximation is
described, for both sets of boundary conditions, in Section \ref{sec:ana}. A-priori
error estimates are proved in Section \ref{sec:error.est}.
The optimality of the predicted rates of convergences is illustrated, for each
boundary condition, in Section \ref{sec:numer} through various numerical
results. 

\section{A mixed formulation of a sixth-order elliptic equation}
\label{SCTformulation}

Let $\Omega \subset \BBR^d$, $d\in \{2,3\}$, be a bounded
domain  with polygonal or polyhedral boundary $\partial\Omega$ and 
outward pointing normal $\bn$ on  $\partial\Omega$.
We consider the sixth-order problem 
\begin{equation}\label{biharm}
-\Delta^3 u =f \quad\text{in}\quad \Omega
\end{equation}
with $f\in H^{-1}(\Omega)$  and two sets of boundary conditions (BCs). 
The first set is the set of \emph{simply supported} boundary conditions 
\begin{equation}\label{sabc}
u =\Delta u=\Delta^2 u=0\quad\text{on}\quad \partial\Omega,
\end{equation}
and the second set is the set of \emph{clamped} boundary conditions
\begin{equation}\label{clabc}
u =\frac{\partial u}{\partial \bn}=\Delta u=0\quad\text{on}\quad \partial\Omega.
\end{equation}


We aim at obtaining a formulation only based on the $H^1$-Sobolev space.
We begin by defining the Lagrange multiplier space:
\\
\begin{itemize}
	\item {\bf Simply supported boundary conditions.} We set
		\[
			\M = H_0^1(\Omega)\,,
		\]
		and equip $\M$ with the norm
		\[
			\|v\|_{\M} = \|v\|_{1,\Omega}\,.
		\]
	\item {\bf Clamped boundary conditions.} We set
		\[
			\M = \{ q \in H^{-1}(\Omega):\, \Delta q \in H^{-1}(\Omega)\}\,,
		\]
		where $\Delta q$ is interpreted in the distributional sense, and the space $\M$ is 
		equipped with the graph norm
		\[
			\|q\|_{\M} = \sqrt{\|q\|_{-1,\Omega}^2 + \|\Delta q\|_{-1,\Omega}^2}.
		\]
		We use the notation  $\langle \cdot,\cdot \rangle$ for the duality pairing 
		between the two spaces $H^1_0(\Omega)$ and $H^{-1}(\Omega)$, so that 
		$\langle u,q\rangle$ and $\langle u,\Delta q\rangle$ are well defined for
		$ u \in H^1_0(\Omega)$ and $q\in  \M$. We note that 
		this space $\M$ is less regular than $H^1(\Omega)$ (see \cite{BGM92,Zul15}).
\end{itemize}
%
%
\medskip

Let $ k \in \BBN\cup \{0\}$. 
We use the standard notations to represent Sobolev spaces \cite{Ada75,BS92}. 
We use $(\cdot,\cdot)_{k,\Omega}$ and $\|\cdot\|_{k,\Omega}$ to denote the 
inner product and norm in $H^k(\Omega)$, respectively. 
When $k=0$, we get the inner product $(\cdot,\cdot)_{0,\Omega}$ and the norm 
$\|\cdot\|_{0,\Omega}$ in $L^2(\Omega)$. The norm of
$W^{k,p}(\Omega)$ is denoted by $\|\cdot\|_{k,p,\Omega}$.

To obtain the $H^1$-based formulation of our boundary value problems, we
introduce an additional unknown $\phi = \Delta u$ and write a weak form of this
equation by formally multiplying by a function $ q \in \M$  and 
integrating over $\Omega$, as in \cite{BGM92,Zul15}.  The variational equation
is now written as 
\[
\langle \phi,q\rangle - \langle u, \Delta q \rangle =0,\quad q \in \M.
\]
Keeping in mind that $u$ will be taken in $H^1_0(\Omega)$, and considering the
BC-dependent $\M$, we see that this variational definition of ``$\phi=\Delta
u$'' also formally imposes the condition $\frac{\partial u}{\partial
\boldsymbol{n}}=0$ on $\partial\Omega$, in the case of clamped BCs. For simply
supported BCs, this does not impose any additional boundary conditions.

To write the mixed formulation in a standard setting, we introduce 
the function space 
$V =H^1_0(\Omega)\times H_0^1(\Omega)$  with the inner product $(\cdot,\cdot)_{V}$ 
defined as \[ ((u,\phi),(v,\psi))_{V} =(\nabla u,\nabla v)_{0,\Omega}+ 
(\nabla \phi,\nabla \psi)_{0,\Omega}\] 
 and with the norm $\|\cdot\|_{V}$ induced by this inner product. 
We now consider the constraint minimisation problem of finding $(u,\phi) \in \CV$  such that 
\begin{equation}\label{cmbiharm}
\CJ(u,\phi)=\inf_{(v,\psi) \in \CV}\CJ(v,\psi),
\end{equation}
where 
\begin{equation}\label{def:CV}
\begin{aligned}
\CJ(v,\psi)={}&\frac{1}{2} \int_{\Omega}|\nabla \psi|^2\,dx-\langle f,v\rangle,\quad\text{and}\\ 
\CV={}&\{(v,\psi)\in V:\;\langle\psi,q\rangle-\langle u, \Delta q \rangle=0,\;
q \in \M\}.
\end{aligned}
\end{equation}
Looking for $(u,\psi)$ in $V$ enables us to account for the conditions
$u=\Delta u=0$ on $\partial\Omega$, valid for both simply supported and clamped BCs.

The problem \eqref{cmbiharm} can be recast as 
a saddle-point formulation: 
find $((u,\phi),\lambda) \in V \times \M$ so that 
\begin{equation}\label{wbiharm}
\begin{array}{llccc}
a((u,\phi),(v,\psi))+&b((v,\psi),\lambda)&=&\ell(v),&\quad (v,\psi) \in  V, \\
b((u,\phi),\mu)&&=&0,&\quad \mu\in \M,
\end{array}
\end{equation}
where
\begin{equation}\label{defbiharm}
\begin{aligned}
&a((u,\phi),(v,\psi))=\int_{\Omega}\nabla\phi\cdot \nabla \psi\,dx,\;\;
b((v,\psi),\mu)=\langle \psi,\mu\rangle-\langle v, \Delta \mu \rangle,\\
&\ell(v)=\langle f,v\rangle.
\end{aligned}
\end{equation}
Using $v=0$ and $\psi\in C^\infty_c(\Omega)$ in the first equation in \eqref{wbiharm}
shows that $\Delta \phi=\lambda$. In the case of simply supported boundary conditions, since
$\lambda\in \Mb=H^1_0(\Omega)$ and $\phi=\Delta u$, this enables us to formally recover the last
missing boundary condition $\Delta^2u=0$ on $\partial \Omega$.

\medskip

The following theorem, whose proof can be found in the appendix, states the well-posedness of our continuous
saddle-point problem.

\begin{theorem}
\label{TMexistence}
There exists a unique $((u,\phi),\lambda) \in V \times \M$ satisfying \eqref{wbiharm}.
\end{theorem}

 
\section{Finite element discretisations}\label{sec:ana}
We consider a quasi-uniform and shape-regular triangulation $\CT_h$ of the 
polygonal domain $\Omega$, where $\CT_h$
consists of triangles, tetrahedra, parallelograms or hexahedra.
Let $S^k_h \subset H^1(\Omega)$ be a standard Lagrange finite element 
space of degree $k\geq 1$ based on the triangulation $\CT_h$ 
with the following approximation property: For $u \in H^{k+1}(\Omega)$
\begin{equation}\label{approx0}
\inf_{v_h \in S^k_h} \left(\|u-v_h\|_{0,\Omega} +h \|u-v_h\|_{1,\Omega} \right)\leq C h^{k+1} \|u\|_{k+1,\Omega}.
\end{equation}
The definition of discrete Lagrange multiplier spaces $\Mb^k_h$ requires some work.
A standard requirement for the construction is the following list of properties:
\begin{enumerate}
\item[[P1\!\!]] $\Mb^k_h \subset H^1(\Omega)$.
\item[[P2\!\!]]  There is a constant $C$ independent of the triangulation such 
that 
\begin{eqnarray*}
\|\theta_h\|_{0,\Omega} \leq C \sup_{\phi_h \in S^k_{h,0}} 
\frac{ \displaystyle\int_{\Omega} \theta_h\phi_h\,dx} {\|\phi_h\|_{0,\Omega}},
\quad \theta_h \in \Mb^k_{h}.
\end{eqnarray*}
\item[[P3\!\!]] There is a constant $C$ independent of the triangulation
such that, if $(u,\phi,\lambda)$ is a solution to \eqref{wbiharm},
$\lambda\in H^k(\Omega)$ and $\mu_h\in \Mb^k_h$ is the
$H^1$-orthogonal projection of $\lambda$ on $\Mb^k_h$, then
\begin{equation}\label{approx1}
\|\lambda-\mu_h\|_{0,\Omega} \leq C h^k \|\lambda\|_{k,\Omega}.
\end{equation}
\end{enumerate}
We now define:
\begin{itemize}
\item {\bf Simply supported boundary conditions.} In this case we may simply take
\[
	S^k_{h,0}= S^k_h\cap H^1_0(\Omega)\,,
	\quad
	V^k_h = S^k_{h,0}\times S^k_{h,0}\,,
	\quad
	\Mb^k_h = S^k_{h,0}
	\,.
\]

The norm on $\Mb_h^k$ is defined by
\[
	\|\mu_h\|_{h} = \sqrt{\|\mu_h\|_{-1,h}^2+ \|\Delta \mu_h\|_{-1,h}^2}\quad\text{with}\quad
	\|\mu_h\|_{-1,h}= \sup_{v_h \in S^k_{h,0}} \frac{\langle \mu_h,v_h\rangle}{\|\nabla {v_h}\|_{0,\Omega}}
\,.
\]
The reader may wish to compare this with \cite{Zul15}, where a similar norm is used, albeit with
$\mu_h\in L^2(\Omega)$.
Properties [P1] and [P2] are trivial.
Property [P3] is established by invoking the fact that $\lambda=0$ on $\partial\Omega$
and using the approximation results in \cite{Bra01,BS94}.

\item {\bf Clamped boundary conditions.} The first two spaces are
\[
	S^k_{h,0}= S^k_h\cap H^1_0(\Omega)\,,
	\quad
	V^k_h = S^{k+1}_{h,0}\times S^k_{h,0}\,,
\]
	however the space $\Mb^k_h$ is not so easily defined.
If we take $\Mb^k_h = S^k_{h,0}$, the Lagrange multiplier
space does not have the required approximation property, due to the constraint
on the boundary condition. On the other hand, if
we take $\Mb^k_h= S^k_h$, the stability assumption 
[P2] will be lost.  

To overcome this, we draw inspiration from the idea used in the mortar
finite element method \cite{LSW05,Lam06}: We construct the Lagrange multiplier
space $\Mb^k_h$ satisfying $\dim \Mb^k_h = \dim S^k_{h,0}$ and the
approximation property \eqref{approx1}.  To construct the basis functions of
$\Mb_h^k$ for the clamped boundary condition we start with $S_h^k$ and remove
all basis functions of $S_h^k$ associated with the boundary of the domain
$\Omega$.  We construct the basis functions of $\Mb^k_h$ according to the following
steps:
\begin{enumerate}
\item For a basis function $\varphi_n$  of $S_h^k$ associated with the point $x_n$ 
on the boundary we find a closest \emph{internal} triangle/tetrahedron/parallelotope $T \in \CT_h$
(this means that $T$ does not touch $\partial\Omega$).
\item The basis functions $\{\varphi_{T, i}\}_{i=1}^m$ associated with internal points of
$T$ can be considered as polynomials defined on the whole domain $\Omega$. Hence, we can compute
$\{\alpha_{T,i} \}_{i=1}^m$  as $\alpha_{T,i}  = \varphi_{T, i}(x_n)$ for $i=1,\cdots,m.$
{This means when computing $\{\alpha_{T,i} \}_{i=1}^m$ 
we regard $  \{\varphi_{T, i}\}_{i=1}^m$ as polynomials with support on $\overline{\Omega}$.}
 For the linear finite element, the coefficients $\{\alpha_{T,i} \}_{i=1}^m$ are the barycentric coordinates 
 of $x_n$ with respect to $T$. 
\item Then we modify all the basis functions  $\{\varphi_{T,i}\}_{i=1}^m$ associated
with $T$ as $ \tilde \varphi_{T,i}= \varphi_{T,i} + \alpha_{T,i} \varphi_n$.
\end{enumerate}
In other words, basis functions associated with boundary points are ``redistributed'' on
basis functions associated with nearby internal points, which ensures that, even after removing
these boundary basis functions, the space $\Mb^k_h$ has the same approximation property
as $S_h^k$.
The norm on $\Mb_h^k$ is defined by
\[ \|\mu_h\|_{h} = \sqrt{\|\mu_h\|_{-1,h}^2+ \|\Delta \mu_h\|_{-1,h*}^2}\quad\text{with}\quad
\|\mu_h\|_{-1,h*}= \sup_{v_h \in S^{k+1}_{h,0}} \frac{\langle \mu_h,v_h\rangle}{\|\nabla {v_h}\|_{0,\Omega}}.\]
Then [P2] and the optimal
approximation property \eqref{approx1} follow (see \cite{Lam06,LSW05}).
\end{itemize}
~\\

In the following, we use a generic constant $C$, which takes 
different values in different occurrences but is always independent of the 
mesh size. 
Now, the finite element problem is to find 
$((u_h,\phi_h),\lambda_h) \in V^k_h \times \Mb^k_h$ so that 
\begin{equation}\label{weakh}
\begin{array}{llccc}
a_h((u_h,\phi_h),(v_h,\psi_h))+&b((v_h,\psi_h),\lambda_h)&=&\ell(v_h),&\quad (v_h,\psi_h) \in  V^k_h, \\
b((u_h,\phi_h),\mu_h)&&=&0,&\quad \mu_h\in \Mb^k_h.
\end{array}
\end{equation}
For simply supported BCs, we can take $a_h=a$.
For the case of clamped boundary conditions, $S^k_{h,0}$ is not contained in
$\Mb^k_h$, and so $a_h(\cdot,\cdot)$ is a stabilised form of the bilinear form $a$.
This allows us to establish coercivity (see the proof of Theorem \ref{TMexistdiscrete} below).
We set $a_h(\cdot,\cdot)$ to be 
\begin{equation}
	\label{EQstabilised}
	a_h((u_h,\phi_h),(v_h,\psi_h)) = a((u_h,\phi_h),(v_h,\psi_h)) + 
					 \int_{\Omega} (\phi_h - \Delta_h u_h) (\psi_h - \Delta_h v_h)\,dx,
\end{equation}
where, for $w\in H^1_0(\Omega)+S^{k+1}_{h,0}$, $\Delta_h w \in S^{k+1}_{h,0}$ is given by
\begin{equation}\label{def.Deltah}
	\int_{\Omega} \Delta_h w\, v_h \,dx = -\int_{\Omega} \nabla w\cdot \nabla v_h \,dx, \quad 
v_h \in S_{h,0}^{k+1}.
\end{equation}

\begin{remark} 
For simply supported BCs, for which $a_h=a$, the saddle-point problem \eqref{weakh} can be, as with the continuous
problem, recast in the form of a constraint minimisation problem: find  $(u_h,\phi_h) \in \CV^k_h$ such that
\begin{equation}\label{cmbiharmh}
\CJ(u_h,\phi_h)=\inf_{(v_h,\psi_h) \in \CV^k_h}\CJ(v_h,\psi_h),
\end{equation}
where $ \CV^k_h $ is a kernel space defined as 
\begin{equation}\label{def:CV}
 \CV^k_h =\{ (v_h, \psi_h) \in V^k_h :\, b((v_h,\psi_h),\mu_h) = 0,\; \mu_h \in \Mb^k_h\}.
\end{equation}
\end{remark}

We now show the existence of a unique solution to \eqref{weakh}.

\begin{theorem}
There exists a unique $(u_h,\phi_h) \in V^k_h$ solution to \eqref{weakh}.
\label{TMexistdiscrete}
\end{theorem}

\begin{proof}
Existence of a unique solution to \eqref{weakh} relies on the same three
properties as in the continuous case, namely:

\begin{enumerate}
\item The bilinear forms $a(\cdot,\cdot)$, $b(\cdot,\cdot)$ and the linear form 
$\ell(\cdot)$ are uniformly continuous on $V^k_h \times V^k_h$, $V^k_h \times \Mb^k_h$ and 
$V^k_h$, respectively. The bilinear form $a_h(\cdot,\cdot)$ is continuous (albeit
not uniformly) on $V^k_h \times V^k_h$. Here, $V_h^k$ is endowed with the norm of $V$, and
$\Mb^k_h$ with its norm $\|\cdot\|_h$.
\item The bilinear form $a_h(\cdot,\cdot)$ is uniformly coercive on the kernel space $\CV^k_h$ defined
by \eqref{def:CV}.  
\item The bilinear form $b(\cdot,\cdot)$ satisfies the following inf--sup condition 
\[ \inf_{\mu_h \in \Mb^k_h} \sup_{(v_h, \psi_h)  \in V^k_h} \frac{b((v_h,\psi_h),\mu_h)} 
{ \|(v_h,\psi_h)\|_{V} \|\mu_h\|_{h}} \geq \tilde \beta, \]
where $\tilde \beta $ is a constant independent of the mesh-size.

\end{enumerate}
Since $V^k_h \subset V$, the uniform continuities of $a(\cdot,\cdot)$ and $\ell(\cdot)$
are trivial. The continuity of $a_h(\cdot,\cdot)$ on the finite dimensional space
$V^k_h$ is obvious. However, since we cannot claim that $\|\Delta_h v\|_{0,\Omega}
\le C\|v\|_{1,\Omega}$ with $C$ independent on $h$, this continuity of $a_h(\cdot,\cdot)$ is
not uniform; this is not required to obtain the existence and uniqueness of a solution to the
scheme, but it will force us to  define a stronger, mesh-dependent norm for the convergence analysis (see
Section \ref{sec:convergence.clamped}).
 The uniform continuity of the bilinear form $b(\cdot, \cdot)$ is
proved as follows.  Note that since $\psi_h \in S_{h,0}^k$ we have from the
definition of $\|\cdot\|_{-1,h}$ - norm 
\begin{eqnarray*}
\|\mu_h\|_{-1,h} \|\nabla\psi_h\|_{0,\Omega} = \sup_{v_h \in S_{h,0}^k}\frac{
\int_{\Omega} v_h\,\mu_h\,dx }{\|\nabla v_h\|_{0,\Omega}} \|\nabla\psi_h\|_{0,\Omega}\geq 
\int_{\Omega} \psi_h\,\mu_h\,dx.
 \end{eqnarray*}
For the simply supported case with $v_h \in  S_{h,0}^k$ we have 
\begin{eqnarray*}
\|\Delta \mu_h\|_{-1,h} \|\nabla v_h\|_{0,\Omega} = \sup_{w_h \in S_{h,0}^k}\frac{
\int_{\Omega} \nabla w_h\cdot \nabla\mu_h\,dx }{\|\nabla w_h\|_{0,\Omega}} \|\nabla v_h\|_{0,\Omega}\geq 
\int_{\Omega} \nabla v_h\cdot \nabla\mu_h\,dx\,, 
 \end{eqnarray*}
whereas for the clamped case with $v_h \in  S_{h,0}^{k+1}$ we have 
\begin{eqnarray*}
\|\Delta \mu_h\|_{-1,h*} \|\nabla v_h\|_{0,\Omega} = \sup_{w_h \in S_{h,0}^{k+1}}\frac{
\int_{\Omega} \nabla w_h\cdot \nabla\mu_h\,dx }{\|\nabla w_h\|_{0,\Omega}} \|\nabla v_h\|_{0,\Omega}\geq 
\int_{\Omega} \nabla v_h\cdot \nabla\mu_h\,dx\,. 
 \end{eqnarray*}
The continuity of $b(\cdot,\cdot)$ follows by writing
\begin{align*}
|b((v_h,\psi_h),\mu_h)| = 
\bigg|\langle\psi_h, \mu_h \rangle - 
\langle v_h,\Delta \mu_h\rangle \bigg|   \leq{}& \bigg|\int_{\Omega} \psi_h\,\mu_h\,dx
+ \int_{\Omega} \nabla v_h\cdot \nabla\mu_h\,dx \bigg|\\
\leq{}& \|(v_h,\psi_h)\|_{V} \|\mu_h\|_{h}. 
\end{align*}
This establishes the first condition.
For the second and third conditions, we now must consider the boundary conditions
separately. 

{\bf Simply supported boundary conditions.}
For $(u_h, \phi_h) \in V^k_h $ satisfying 
\[
b((u_h,\phi_h),\mu_h) = 0, \quad \mu_h \in \Mb^k_h,
\]
since $\Mb^k_h=S_{h,0}^k$, we can take $\mu_h=u_h$ to obtain
\[ \int_{\Omega} \nabla u_h \cdot \nabla u_h \,dx = - \int_{\Omega} \phi_h u_h\,dx.\]
Hence, using the Cauchy--Schwarz and Poincar\'e inequalities we obtain 
\[ \|\nabla u_h \|^2_{0,\Omega}  \leq C_1 \|\phi_h\|_{0,\Omega} \|\nabla u_h\|_{0,\Omega}.\]
The coercivity then follows exactly as in the continuous case:
\[ \|\nabla u_h \|^2_{0,\Omega}+  \|\nabla\phi_h\|^2_{0,\Omega}
\leq  C a((u_h,\phi_h),(u_h,\phi_h)),\quad 
(u_h,\phi_h) \in \CV^k_h.
\]

For the inf--sup condition we set $\psi_h =0$  as in the continuous setting to obtain 
\[  \sup_{(v_h, \psi_h)  \in V^k_h} \frac{b((v_h,\psi_h),\mu_h)} 
{ \|(v_h,\psi_h)\|_{V} } {\ge}\sup_{v_h\in S^k_{h,0}}
\frac{ \langle v_h,\Delta \mu_h\rangle} 
{ \|\nabla v_h\|_{0,\Omega} } \geq \|\Delta \mu_h\|_{-1,h},\]
and setting $v_h=0$ to find 
\[  \sup_{(v_h, \psi_h)  \in V^k_h} \frac{b((v_h,\psi_h),\mu_h)} 
{ \|(v_h,\psi_h)\|_{V} }  {\ge}\sup_{ \psi_h  \in S^k_{h,0}} 
\frac{\langle\psi_h, \mu_h \rangle } { \|\nabla \psi_h\|_{0,\Omega} } \geq \|\mu_h\|_{-1,h}. \]
Thus \[  \sup_{(v_h, \psi_h)  \in V^k_h} \frac{b((v_h,\psi_h),\mu_h)} 
{ \|(v_h,\psi_h)\|_{V} } 
\geq  \tilde  \beta \|\mu_h\|_{h}.\]

{\bf Clamped boundary conditions.}
Recalling the stabilisation term in $a_h(\cdot,\cdot)$, we use the Poincar\'e inequality for $u_h \in S^{k+1}_{h,0}$
and the definition \eqref{def.Deltah} of $\Delta_h$ to find
\begin{align*}
	\|\nabla u_h \|_{0,\Omega}  = \sup_{v_h \in S^{k+1}_{h,0}} \frac{\int_{\Omega} \nabla u_h \cdot \nabla v_h\, dx}
{\|\nabla v_h\|_{0,\Omega}} \leq{}&  C \sup_{v_h \in S^{k+1}_{h,0}} \frac{\int_{\Omega} \nabla u_h \cdot \nabla v_h\, dx}
{\| v_h\|_{0,\Omega}}\\
={}&C \sup_{v_h \in S^{k+1}_{h,0}} \frac{-\int_{\Omega} \Delta_h u_h\, v_h\, dx}
{\| v_h\|_{0,\Omega}} \leq  C\|\Delta_h u_h\|_{0,\Omega}.
\end{align*}
Hence, using Poincar\'e inequality again, there exists a positive constant $C$ such that,
for all $\phi_h \in S_{h,0}^k$,
\[
\|\nabla u_h \|^2_{0,\Omega}  \leq C \left(\|\phi_h - \Delta_h u_h\|^2_{0,\Omega} + \|\phi_h\|^2_{0,\Omega}\right)
\leq  C \left(\|\phi_h - \Delta_h u_h\|^2_{0,\Omega} + \|\nabla \phi_h\|^2_{0,\Omega}\right).
\]
Thus we have the coercivity of the modified bilinear form $a_h(\cdot,\cdot)$ on $S_{h,0}^{k+1} \times S_{h,0}^k$ 
and, hence, on the discrete kernel space $\CV_h^k \subset S_{h,0}^{k+1} \times S_{h,0}^k$ 
with respect to the standard norm of $V$.
The inf--sup condition now follows as in the case of simply supported boundary
conditions, with $S_{h,0}^{k+1}$ instead of $S_{h,0}^k$ for $v_h$, which
accounts for $\|\cdot\|_{-1,h*}$ used in the definition of the norm on $\Mb_h^k$.
This finishes the proof of the theorem.
\end{proof}

\section{\emph{A priori} error estimates}\label{sec:error.est}
In this section we investigate \emph{A priori} error estimates for our problems.

\subsection{\emph{A priori} error estimate for simply supported boundary conditions}
Our goal is to establish the following theorem.
\begin{theorem}\label{th0}
Let $(u,\phi,\lambda)$ be the solution of the saddle-point problem 
\eqref{wbiharm}, and $(u_h,\phi_h,\lambda_h)$ the solution of \eqref{weakh},
both with simply supported boundary conditions.
We assume that $u,\phi\in H^{k+1}(\Omega)$ and 
$\lambda\in H^k(\Omega)$.
 Then 
\begin{equation} \label{eq:strang1}
\|(u-u_h,\phi-\phi_h)\|_{V} \leq 
C  h^{k} \left(\|u\|_{k+1,\Omega} + \|\phi\|_{k+1,\Omega}+ |\lambda|_{k,\Omega}\right).
\end{equation}
\end{theorem}

To prove this theorem we apply Strang's second lemma \cite{BS94}:
\begin{multline} \label{eq:strang}
\|(u-u_h,\phi-\phi_h)\|_{V} \\
\leq
C\left(\inf_{(v_h,\psi_h) \in \CV^k_h} \|(u-v_h,\phi-\psi_h)\|_{V}
+\sup_{(v_h,\psi_h) \in \CV^k_h} \frac{|a((u-u_h,\phi-\phi_h), (v_h,\psi_h))|}
{\|(v_h,\psi_h)\|_{V}}\right),
\end{multline}
where $(u,\phi)$ is the solution of \eqref{cmbiharm},
 and $(u_h,\phi_h)$ the solution of \eqref{weakh}
(recall that, here, $a_h=a$). 
The first term in the right side of \eqref{eq:strang} is the 
best approximation error and the second one stands for the consistency error. 
First we turn our attention to this latter term.
 
\begin{lemma}\label{lemma2} 
Let $(u,\phi,{\lambda})$ be the solution of the saddle-point problem 
\eqref{wbiharm} with simply supported boundary conditions.  Then, if $\lambda \in H^k(\Omega)$, we have 
\[ \sup_{(v_h,\psi_h) \in \CV^k_h} \frac{|a((u-u_h,\phi-\phi_h), (v_h,\psi_h))|}
{\|(v_h,\psi_h)\|_{V}} \leq C h^k |\lambda|_{k,\Omega}. \]
\end{lemma}
\begin{proof}
From the first equation of  \eqref{wbiharm} we  get
 $a((u-u_h,\phi-\phi_h),(v_h,\psi_h))+ b((v_h,\psi_h),\lambda)=0$ for all 
$(v_h,\psi_h) \in \CV^k_h$. Hence, 
\begin{eqnarray*} 
\sup_{(v_h,\psi_h) \in \CV^k_h} \frac{|a((u-u_h,\phi-\phi_h), (v_h,\psi_h))|}
{\|(v_h,\psi_h)\|_{V}} =
\sup_{(v_h,\psi_h) \in \CV^k_h} \frac{|b( (v_h,\psi_h),\lambda)|}
{\|(v_h,\psi_h)\|_{V}}.
\end{eqnarray*}
Denoting the projection of $\lambda$ onto $\Mb^k_h=S_{h,0}^k$ with respect to the
$H^1$-inner product by {$\tilde\lambda_h$}, we have 
\begin{equation}\label{estorth}
 \int_{\Omega} \nabla v_h \cdot \nabla (\lambda-\tilde \lambda_h)\,dx = 
-\int_{\Omega} v_h (\lambda-\tilde \lambda_h)\,dx.
\end{equation}
As $(v_h,\psi_h) \in \CV^k_h$, using \eqref{estorth},
\[ b((v_h,\psi_h),\lambda)= b((v_h,\psi_h),\lambda-\tilde \lambda_h)= 
-\int_{\Omega} v_h (\lambda-\tilde \lambda_h)\,dx + 
\int_{\Omega} \psi_h (\lambda-\tilde \lambda_h)\,dx,
\]
and [P3] thus yields
\[ |b((v_h,\psi_h),\lambda)|\leq C h^k |\lambda|_{k,\Omega}\,\|(v_h, \psi_h)\|_{V}.\]
Thus 
\begin{eqnarray*}
\sup_{(v_h,\psi_h) \in \CV^k_h} \frac{|a((u-u_h,\phi-\phi_h), (v_h,\psi_h))|}
{\|(v_h,\psi_h)\|_{V}} 
&\leq& C h^k |\lambda|_{k,\Omega}.
\end{eqnarray*}
\end{proof}

We now prove the following lemma, which is 
similar  to \cite[Proposition~3]{DP01}. See also \cite{Lam11a}.

\begin{lemma}\label{lemma3}
Let $(w_h,\xi_h) \in \CV^k_h$, $(w,\xi)\in  \CV$, and 
$R_h^k:H_0^1(\Omega)\rightarrow S^k_{h,0}$ be 
the Ritz projector (also called ``elliptic projector'') defined as 
\[\int_{\Omega}\nabla (R_h^k w-w)\cdot\nabla v_h\,dx=0,\; v_h \in S^k_{h,0}.
\] 
 Then 
\[ |w-w_h|_{1,\Omega} \leq C \|\xi-\xi_h\|_{0,\Omega}+|R_h^kw-w|_{1,\Omega}. \]
\end{lemma}
\begin{proof}
Here we have \begin{eqnarray*}
\int_{\Omega}\nabla w\cdot\nabla q+\xi\,q\,dx=0,\; q\in H^1_0(\Omega),\;\text{and}\;
\int_{\Omega}\nabla w_h\cdot\nabla q_h+\xi_h\,q_h\,dx=0,\; q_h\in S^k_{h,0},
\end{eqnarray*} 
since $(w_h,\xi_h) \in \CV^k_h$ and $(w,\xi)\in  \CV$.
Thus, since $S_{h,0}^k\subset H^1_0(\Omega)$,
\begin{equation}\label{eq2}
\int_{\Omega}\nabla (w-w_h)\cdot\nabla q_h+(\xi-\xi_h)\,q_h\,dx=0,\; 
q_h\in S^k_{h,0}.
\end{equation}
In terms of the Ritz projector $R_h^k$, \eqref{eq2} is written as 
\begin{equation}\label{eq3}
\int_{\Omega}\nabla (R_h^kw-w_h)\cdot\nabla q_h+(\xi-\xi_h)\,q_h\,dx=0,\; 
q_h\in S^k_{h,0}.
\end{equation}
Taking $q_h = R_h^kw-w_h$ in equation \eqref{eq3} 
and using the Cauchy--Schwarz and Poincar\'e inequalities, we obtain 
\[
|R_h^k w - w_h|^2_{1,\Omega}\leq \|\xi-\xi_h\|_{0,\Omega}\|R_h^k w - w_h\|_{0,\Omega}
\leq C\|\xi-\xi_h\|_{0,\Omega}|R_h^k w - w_h|_{1,\Omega},
\]
which yields $|R_h^k w - w_h|_{1,\Omega}\leq C\|\xi-\xi_h\|_{0,\Omega}$.
The final result follows from the triangle inequality 
\begin{equation*}
|w-w_h|_{1,\Omega}\leq |R_h^k w - w_h|_{1,\Omega}+|w-R_h^kw|_{1,\Omega} \leq 
  C\|\xi-\xi_h\|_{0,\Omega}+ |w-R_h^kw|_{1,\Omega}.
\end{equation*}
\end{proof}

The following lemma estimates the best approximation error in \eqref{eq:strang}, and concludes
the proof of Theorem \ref{th0}.

\begin{lemma}\label{lemma12}
For any $(u,\phi) \in \CV\cap (H^{k+1}(\Omega)\times H^{k+1}(\Omega))$, there exists $(w_h,\psi_h) \in \CV^k_h$ such that
\begin{equation}\label{eq0}
\|(u-w_h,\phi-\xi_h)\|_{V} \leq  Ch^k\left(  \|u\|_{k+1,\Omega} +  \|\phi\|_{k+1,\Omega} \right)
\end{equation}
\end{lemma}

\begin{proof}
Let $\Pi_h^k:L^2(\Omega) \to S^k_{h,0}$ be the orthogonal projector onto $S^k_{h,0}$. 
Let $(w_h,\xi_h) \in V^k_h$ be defined as 
\begin{eqnarray*}
\int_{\Omega} (\phi -\xi_h)\,q_h \,dx = 0,\; q_h \in S^k_{h,0},\;\text{and}\;
\int_{\Omega} \nabla w_h\cdot\nabla q_h + \xi_h\,q_h\,dx = 0,\; q_h \in S^k_{h,0}.
\end{eqnarray*}
Hence, $(w_h,\xi_h) \in \CV^k_h$ with $\xi_h =\Pi_h^k \phi$.
Moreover, since $\Pi_h^k$ is 
the $L^2$-projector onto $S^k_{h,0}$ we have \cite{Bra01}
\[
|\phi - \xi_h|_{1,\Omega} 
\leq C h^k |\phi|_{k+1,\Omega}.
\]
We note that the Ritz projector $R_h^k$ as defined in Lemma \ref{lemma3}
has the approximation property \cite{Tho97} \[
|u-R_h^k u|_{1,\Omega}\leq C h^k |u|_{k+1,\Omega}.\]
Hence, using the result of Lemma \ref{lemma3} 
we obtain 
\[|u-w_h|_{1,\Omega}\leq \|\phi - \xi_h\|_{0,\Omega}+ |u-R_h^ku|_{1,\Omega}
\leq C h^k( |u|_{k+1,\Omega}+ |\phi|_{k+1,\Omega}).
\]
\end{proof}

\subsection{\emph{A priori} error estimates for clamped boundary conditions}\label{sec:convergence.clamped}

The error estimates for clamped boundary conditions are established in the following mesh-dependent
semi-norm: for $(u,\phi)\in V+V_h^k$,
\begin{equation}\label{norm}
\snorm{(u, \phi)}{k,h}  = \sqrt{ \|\nabla \phi\|^2_{0,\Omega}  + \|\phi  -\Delta_h u \|^2_{0,\Omega} }.
\end{equation}
The reason for introducing this semi-norm is that, as already noticed in the proof of Theorem \ref{TMexistdiscrete}, 
the stabilisation term in $a_h(\cdot,\cdot)$ is not uniformly continuous on $V_h^k$ for the norm of $V$.
On the contrary, $a_h(\cdot,\cdot)$ is uniformly continuous for $\snorm{\cdot}{k,h}$, which enables the usage
of the second Strang Lemma.

Our goal here is to establish the following \emph{A priori} estimate.

\begin{theorem}\label{th1}
Let $(u,\phi,\lambda)$ be the solution of the saddle-point problem 
\eqref{wbiharm}, and $(u_h,\phi_h,\lambda_h)$ the solution of \eqref{weakh}, both with
clamped boundary conditions.
We assume that $u\in W^{k+1,p}(\Omega)$ for some $p\ge 2$, $\phi \in H^{k+1}(\Omega)$ and that
$\lambda\in H^k(\Omega)$. 
We have  
\begin{equation} \label{eq:strangn2}
\snorm{(u-u_h,\phi-\phi_h)}{k,h}
\leq 
C \left( h^{k} \|u\|_{k+1,\Omega} + 
h^{k-\frac{1}{2} - \frac{1}{p}} \|u\|_{k+1,p,\Omega}+h^k\|\phi\|_{k+1,\Omega}+ h^k |\lambda|_{k,\Omega}\right).
\end{equation}
\end{theorem}

\begin{remark}
Due to the uniform coercivity property of $a_h(\cdot,\cdot)$ on $\CV^k_h$ (see Theorem \ref{TMexistdiscrete}),
$\snorm{\cdot}{k,h}$ is a norm that is uniformly stronger than the $V$ norm, i.e.,
there is $C>0$ independent of $h$ such that, if $( u_h,\phi_h)\in \CV^k_h$ then 
\[
C\snorm{(u_h,\phi_h)}{k,h} \geq  \|\nabla \phi_h \|^2_{0,\Omega}  + \|\nabla u_h \|^2_{0,\Omega}.
\]
This property is all that is required to apply the second Strang lemma below. The semi-norm is not
a norm on $V$, but the following property can be established:
the kernel of $\snorm{\cdot}{k,h}$ consists of pairs $(u,0)$ such that
\[ \int_{\Omega} \nabla u \cdot \nabla u_h\,dx = 0,\quad u_h \in S_{h,0}^{k+1}.\]
Hence, even though the estimate \eqref{eq:strangn2}
might not `capture' a part of the solution $(u,\phi)$, that part actually converges to zero in $L^2$- and $H^1$-norms.
\end{remark}

We follow a strategy analogous to that used for simply supported BCs. 
Even though the second Strang lemma is often used for bilinear
forms $a_h(\cdot,\cdot)$ that are coercive on the entire continuous and discrete spaces,
the proof of \cite[Lemma 1.2, Chap. III, \S{} 1]{Bra01} and the uniform
coercivity (by construction) of $a_h(\cdot,\cdot)$ with respect to $\snorm{\cdot}{k,h}$ show that the following estimate holds:
\begin{multline} \label{eq:strang2}
\snorm{(u-u_h,\phi-\phi_h)}{k,h} \\
\leq 
C\left(\inf_{(v_h,\psi_h) \in \CV^k_h} \snorm{(u-v_h,\phi-\psi_h)}{k,h}
+\sup_{(v_h,\psi_h) \in \CV^k_h} \frac{|a_h((u,\phi), (v_h,\psi_h))-\ell(v_h)|}
{\snorm{(v_h,\psi_h)}{k,h}}\right).
\end{multline}
Theorem \ref{th1} is proved if
we bound the right-hand side of the above inequality by the
right-hand side of \eqref{eq:strangn2}.

First we prove the following lemma to estimate  the consistency error term
\[ \sup_{(v_h,\psi_h) \in \CV^k_h} \frac{|a_h((u,\phi), (v_h,\psi_h))-\ell(v_h)|}
{\snorm{(v_h,\psi_h)}{k,h}}.\]
\begin{lemma}\label{lemma5} 
Let $(u,\phi,{\lambda})$ be the solution of the saddle-point problem 
\eqref{wbiharm}.  Then, if $\lambda \in H^k(\Omega)$, $\phi \in H^k(\Omega)$  and $ u \in H^2(\Omega)$ we have 
\[ \sup_{(v_h,\psi_h) \in \CV^k_h} \frac{|a_h((u,\phi), (v_h,\psi_h))-\ell(v_h)|}
{\snorm{(v_h,\psi_h)}{k,h}}\leq C h^k \left(|\lambda|_{k,\Omega} +|\phi|_{k,\Omega} \right). \]
\end{lemma}
\begin{proof}
Here 
\begin{align*}
a_h((u,\phi), (v_h,\psi_h))-\ell(v_h) ={}& a((u,\phi), (v_h,\psi_h))\\
&+ \int_{\Omega} (\phi - \Delta_h u) (\psi_h -\Delta_h v_h) \,dx
-\ell(v_h).
\end{align*}
The first equation of  \eqref{wbiharm} yields
\[ a((u,\phi), (v_h,\psi_h)) + b((v_h,\psi_h),\lambda) = \ell(v_h) ,\quad (v_h,\psi_h) \in V^k_h,\]
and thus 
\[ a_h((u,\phi), (v_h,\psi_h))-\ell(v_h) =   \int_{\Omega} (\phi - \Delta_h u) (\psi_h -\Delta_h v_h) \,dx-b((v_h,\psi_h),\lambda).\]
The term $b((v_h,\psi_h),\lambda)$ can be estimated as in Lemma \ref{lemma2}. 
The stabilisation term is easily bounded using the Cauchy--Schwarz inequality 
\[ \int_{\Omega} (\phi - \Delta_h u) (\psi_h -\Delta_h v_h) \,dx \leq 
\|\phi - \Delta_h u\|_{0,\Omega} \|\psi_h - \Delta_h v_h\|_{0,\Omega}.\]
We further note that, for $u \in H^2(\Omega)$,
\[ \int_{\Omega} \Delta_h u v_h\,dx = - \int_{\Omega} \nabla u \cdot \nabla v_h \,dx = 
\int_{\Omega} \Delta u v_h\,dx,\quad v_h \in S_{h,0}^{k+1},\]
and thus 
\[ \Delta_h u = \Pi_h^{k+1} \phi,\]
where $\Pi_h^{k+1}$ is the $L^2(\Omega)$-orthogonal projector onto $S^{k+1}_{h,0}$.
The proof follows using the approximation property \eqref{approx0} on
$S^{k+1}_{h}$. 
\end{proof}

The following lemma estimates the best approximation error in the mesh-depen\-dent norm.

\begin{lemma}\label{lemma0}
Let $(u,\phi) \in \CV$ with 
$u \in W^{k+1,p}(\Omega)$ (for $p \geq {2}$) and
$\phi \in H^{k+1}(\Omega)$. Then, there exists an element $(w_h,\psi_h) \in \CV^k_h$
such that
\begin{equation}\label{eqn0}
\snorm{(u-w_h,\phi-\psi_h)}{k,h} \leq  C\left( h^{k} \|u\|_{k+1,\Omega} + h^k \|\phi\|_{k+1,\Omega} + 
h^{k-\frac{1}{2} - \frac{1}{p}} \|u\|_{k+1,p,\Omega}\right).
\end{equation}
\end{lemma}
\begin{proof}
We start with the definition of the mesh-dependent norm 
\[ \snorm{(u-w_h,\phi-\psi_h)}{k,h}^2 = \|\nabla (\phi - \psi_h)\|^2_{0,\Omega} + 
\|\phi-\psi_h - \Delta_h (u-w_h)\|^2_{0,\Omega}.\]

Let $R_h^{k+1}:{H^1_0}(\Omega)\rightarrow S^{k+1}_{h,0}$ be 
the Ritz projector defined for $w \in H^1_0(\Omega)$
\[\int_{\Omega}\nabla (R_h^{k+1} w-w)\cdot\nabla v_h\,dx=0,\; v_h \in S^{k+1}_{h,0}.
\] 
With $w_h = R_h^{k+1} u$, Property [P2] enables us to define  $\psi_h \in S^k_{h,0}$ by 
\begin{equation*}
\int_{\Omega}\nabla w_h\cdot\nabla \mu_h+\psi_h\,\mu_h\,dx=0,\; 
\mu_h\in \Mb^k_h.
\end{equation*}
Hence, $(w_h,\psi_h) \in \CV^k_h$ and, since $(u,\phi) \in \CV$ and $\Mb_h^k\subset \M$,
we obtain 
\begin{equation}\label{leq1}
\int_{\Omega}\nabla (u-w_h)\cdot\nabla \mu_h+(\phi-\psi_h)\,\mu_h\,dx=0,\; 
\mu_h\in \Mb^k_h.
\end{equation}
We now use a triangle inequality  to write 
\begin{eqnarray*}
 \snorm{(u-w_h,\phi-\psi_h)}{k,h}^2 &=& 
\|\phi-\psi_h-\Delta_h(u-w_h)\|^2_{0,\Omega}+ |\phi-\psi_h|^2_{1,\Omega}\\ 
&\leq & \|\phi-\psi_h\|^2_{1,\Omega}+\|\Delta_h(u-w_h)\|^2_{0,\Omega}\\
&= & 
\|\Delta_h(u-w_h)\|^2_{0,\Omega}+ \|\phi-Q_h\phi\|^2_{1,\Omega}+\|Q_h\phi-\psi_h\|^2_{1,\Omega},
\end{eqnarray*}
where $Q_h$ is a quasi-projection operator onto $S^k_{h,0}$ defined by 
\begin{equation*}
\int_{\Omega}Q_h \phi \,\mu_h\,dx =\int_{\Omega}\phi\,\mu_h\,dx,\; 
\mu_h\in \Mb^k_h.
\end{equation*}
As above, $Q_h$ is well-defined due to Assumption [P2].
First we estimate the term $\|\Delta_h(u-w_h)\|_{0,\Omega}$.
By definition \eqref{def.Deltah} of $\Delta_h$ and by choice $w_h=R_h^{k+1} u$,
\begin{eqnarray*}
\|\Delta_h(u-w_h)\|_{0,\Omega} &=& 
\sup_{v_h \in S^{k+1}_{h,0}} \frac{\int_{\Omega}  \Delta_h(u-w_h) v_h\,dx}{\|v_h\|_{0,\Omega}}\\
&=& \sup_{v_h \in S^{k+1}_{h,0}} \frac{-\int_{\Omega}  \nabla (u-w_h) \cdot \nabla v_h\,dx}{\|v_h\|_{0,\Omega}}
=0.
\end{eqnarray*}
We know  that $Q_h \phi$ \cite{Lam06,LSW05} has the desired approximation property
\[ 
 |\phi-Q_h\phi|_{1,\Omega} \leq C h^k |\phi|_{k+1,\Omega}. 
\]

Hence, we are left with the term $\|Q_h\phi-\psi_h\|_{1,\Omega}$.
We start with an inverse estimate and use Assumption [P2] and \eqref{leq1} to get 
\begin{eqnarray*}
 \|\psi_h-Q_h\phi\|_{1,\Omega} \leq \frac{C}{h}  \|\psi_h-Q_h\phi\|_{0,\Omega} 
 &\leq& \frac{C}{h}\sup_{\mu_h \in {\Mb^k_h}}
  \frac{\int_{\Omega} (\psi_h-Q_h \phi)\,\mu_h\,dx}{\|\mu_h\|_{0,\Omega}}\\ &\leq& 
\frac{C}{h}\sup_{\mu_h \in {\Mb^k_h}}\frac{\int_{\Omega} (\psi_h- \phi)\,\mu_h\,dx}{\|\mu_h\|_{0,\Omega}}\\ &\leq& 
\frac{C}{h}\sup_{\mu_h \in {\Mb^k_h}} \frac{\int_{\Omega} \nabla (u-w_h)\cdot\nabla \mu_h\,dx}{\|\mu_h\|_{0,\Omega}}.
\end{eqnarray*}
Since $w_h$ is the Ritz projection of $u$ onto $S^{k+1}_{h,0}$, 
the final result follows using 
Lemma \ref{lemma1}. \end{proof}

The following lemma is proved using the ideas in \cite[Lemma 3.2]{GR86}.
See also  \cite{Sch78}. 
\begin{lemma}\label{lemma1}
Let $k \in \BBN$ and $p \in \BBR$ such that $k\geq 1$ 
and $2\leq p \leq \infty$. Let $R_h^{k+1}:H^1_0(\Omega)\to S_{h,0}^{k+1}$ be the Ritz projection as defined in 
Lemma \ref{lemma0}. {Then, there exists a constant $C>0$
such that, for any $ w\in W^{k+1,p}(\Omega)\cap H_0^1(\Omega)$,}
\begin{equation}\label{lest}
\sup_{\mu_h \in \Mb^k_h}\frac{\int_{\Omega}\nabla (w-R_h^{k+1}w)\cdot\nabla \mu_h\,dx}
{\|\mu_h\|_{0,\Omega}} \leq C h^{k+\frac{1}{2}-\frac{1}{p}}\|w\|_{k+1,p,\Omega}. 
 \end{equation}
\end{lemma}

\begin{proof}
Let $\CT_h^1$ be the set of elements in $\CT_h$ touching the boundary of $\Omega$. 
Let $\mu_h \in \Mb_h^k$ be arbitrary and $m_h\in S_{h,0}^k$ which coincides with $\mu_h$ at all interior finite element nodes. 
Since $ S_{h,0}^k \subset S_{h,0}^{k+1}$, we have 
\[ \int_{\Omega} \nabla (w-R_h^{k+1} w)\cdot\nabla \psi_h\,dx =0,\quad \psi_h \in S_{h,0}^k.\]
Thus we have 
\[ \int_{\Omega} \nabla (w-R_h^{k+1} w)\cdot\nabla \mu_h\,dx = 
\sum_{T \in \CT_h^1} \int_{T} \nabla (w-R_h^{k+1} w)\cdot\nabla (\mu_h-m_h) \,dx.\]
The rest of the proof is exactly as in \cite[Lemma 3.2]{GR86}.
\end{proof}

\section{Numerical Results}\label{sec:numer}
In this section, we show some numerical experiments for the sixth-order elliptic equation using both types of boundary conditions. We compute the convergence rates in $L^2$-norm and $H^1$-seminorm for $u$ and $\phi$, and the convergence rates in $L^2$-norm for our Lagrange multiplier. This computation will be done using linear and quadratic finite element spaces.

\begin{figure}[htb]
	\centering
	\begin{subfigure}{.3\textwidth}
		\centering
		\includegraphics[width=\linewidth]{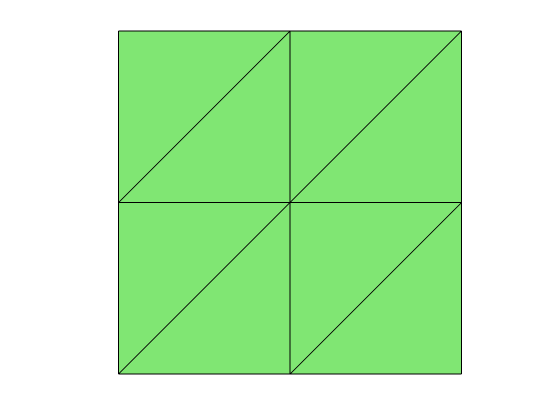}
		\caption{Initial mesh for simply supported boundary conditions}
		\label{fig:init1}
	\end{subfigure} 
	\hspace{2cm}
	\begin{subfigure}{.4\textwidth}
		\centering
		\includegraphics[width=\linewidth]{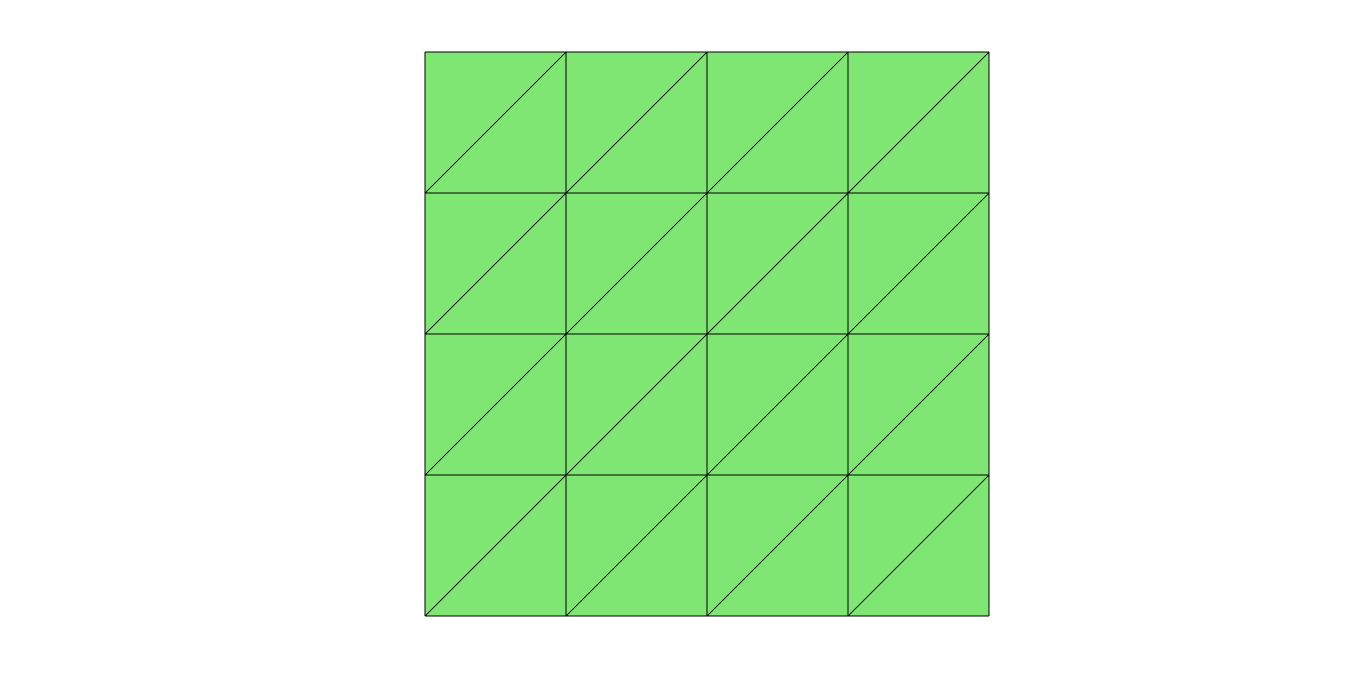}
		\caption{Initial mesh for clamped boundary conditions}
		\label{fig:init2}
	\end{subfigure}
	\caption{Initial meshes}
	\label{imesh}
\end{figure}

\setlength\tabcolsep{4 pt}
\subsection{Simply supported boundary conditions}
\paragraph{Examples 1 and 2} 
We consider the exact solution
\begin{equation}\label{ex.sol.1}
u = x^5(1-x)^5y^5(1-y)^5 \text{ in $\Omega = \left(0,1\right)^2$},
\end{equation}
for the first example and  the exact solution 
\begin{equation}\label{ex.sol.2}
u = \left(e^y + e^x\right)x^5(1-x)^5y^5(1-y)^5 \text{ in $\Omega = \left(0,1\right)^2$},
\end{equation}
for the second example, where both functions satisfy simply supported boundary conditions $u = \Delta u = \Delta^2 u = 0$ on $\partial\Omega$. 
We start with the initial mesh as given in the left picture of Figure \ref{imesh} and 
compute the relative errors in various norms associated with our variables at each step of refinement.

\begin{table}[htb!]
	\centering
	\caption{Discretisation errors  for the simply supported boundary condition: linear case and exact solution \eqref{ex.sol.1}}
	\begin{tabular}{|c|cc|cc|cc|cc|cc|}
		\hline
		\multirow{2}{*}{elem} & \multicolumn{2}{c|}{$\frac{\left\Vert u - u_h\right\Vert_{0,\Omega}}{\left\Vert u\right\Vert_{0,\Omega}}$} & \multicolumn{2}{c|}{$\frac{| u - u_h|_{1,\Omega}}{|u|_{1,\Omega}}$} & \multicolumn{2}{c|}{$\frac{\left\Vert \phi - \phi_h\right\Vert_{0,\Omega}}{\left\Vert \phi\right\Vert_{0,\Omega}}$} & \multicolumn{2}{c|}{$\frac{|\phi - \phi_h|_{1,\Omega}}{| \phi|_{1,\Omega}}$} & \multicolumn{2}{c|}{$\frac{\left\Vert \lambda - \lambda_h\right\Vert_{0,\Omega}}{\left\Vert \lambda\right\Vert_{0,\Omega}}$}\\ \hhline{~----------}
		& error & rate & error & rate & error & rate & error & rate & error & rate \\\hline
   8 & 1.71e+02 &          & 1.44e+02 &          & 8.51e+01 &          & 4.33e+01 &          & 1.88e+01 &          \\\hline
   32 & 5.17e+01 & 1.72 & 3.71e+01 & 1.96 & 1.85e+01 & 2.20 & 8.14 & 2.41 & 3.15 & 2.58 \\\hline
  128 & 1.74e+01 & 1.57 & 1.18e+01 & 1.65 & 5.59 & 1.72 & 2.40 & 1.76 & 1.00 & 1.65 \\\hline
  512 & 4.71 & 1.88 & 3.17 & 1.90 & 1.48 & 1.92 & 6.74e$-$01 & 1.83 & 2.86e$-$01 & 1.81 \\\hline
 2048 & 1.20 & 1.97 & 8.11e$-$01 & 1.97 & 3.76e$-$01 & 1.98 & 2.05e$-$01 & 1.72 & 7.48e$-$02 & 1.94 \\\hline
 8192 & 3.02e$-$01 & 1.99 & 2.08e$-$01 & 1.96 & 9.44e$-$02 & 1.99 & 7.65e$-$02 & 1.42 & 1.89e$-$02 & 1.98 \\\hline
32768 & 7.57e$-$02 & 2.00 & 5.59e$-$02 & 1.89 & 2.36e$-$02 & 2.00 & 3.42e$-$02 & 1.16 & 4.75e$-$03 & 2.00 \\\hline
131072 & 1.89e$-$02 & 2.00 & 1.74e$-$02 & 1.69 & 5.91e$-$03 & 2.00 & 1.65e$-$02 & 1.05 & 1.19e$-$03 & 2.00 \\\hline
524288 & 4.73e$-$03 & 2.00 & 6.75e$-$03 & 1.37 & 1.48e$-$03 & 2.00 & 8.20e$-$03 & 1.01 & 2.94e$-$04 & 2.00 \\\hline
		
	\end{tabular}
	\label{ss1:linear}
\end{table}

\begin{table}[htb!]
	\centering
	\caption{Discretisation errors  for the simply supported boundary condition: quadratic case and exact solution \eqref{ex.sol.1}}
	\begin{tabular}{|c|cc|cc|cc|cc|cc|}
		\hline
		\multirow{2}{*}{elem} & \multicolumn{2}{c|}{$\frac{\left\Vert u - u_h\right\Vert_{0,\Omega}}{\left\Vert u\right\Vert_{0,\Omega}}$} & \multicolumn{2}{c|}{$\frac{| u - u_h|_{1,\Omega}}{|u|_{1,\Omega}}$} & \multicolumn{2}{c|}{$\frac{\left\Vert \phi - \phi_h\right\Vert_{0,\Omega}}{\left\Vert \phi\right\Vert_{0,\Omega}}$} & \multicolumn{2}{c|}{$\frac{|\phi - \phi_h|_{1,\Omega}}{| \phi|_{1,\Omega}}$} & \multicolumn{2}{c|}{$\frac{\left\Vert \lambda - \lambda_h\right\Vert_{0,\Omega}}{\left\Vert \lambda\right\Vert_{0,\Omega}}$}\\ \hhline{~----------}
		& error & rate & error & rate & error & rate & error & rate & error & rate \\\hline
  8 & 1.01e$-$01 &   & 4.70e$-$01 &   & 2.13 &          & 1.00e+01 &          & 4.85e+01 &          \\\hline
   32 & 3.56e$-$04 & 8.15 & 3.32e$-$03 & 7.15 & 3.10e$-$02 & 6.10 & 4.27e$-$01 & 4.55 & 4.72 & 3.36 \\\hline
  128 & 5.55e$-$05 & 2.68 & 6.31e$-$04 & 2.40 & 4.24e$-$03 & 2.87 & 9.65e$-$02 & 2.15 & 7.38e$-$01 & 2.68 \\\hline
  512 & 4.15e$-$06 & 3.74 & 1.33e$-$04 & 2.24 & 3.63e$-$04 & 3.55 & 2.43e$-$02 & 1.99 & 8.15e$-$02 & 3.18 \\\hline
 2048 & 2.89e$-$07 & 3.84 & 3.29e$-$05 & 2.02 & 3.13e$-$05 & 3.54 & 6.26e$-$03 & 1.96 & 9.36e$-$03 & 3.12 \\\hline
 8192 & 2.24e$-$08 & 3.69 & 8.23e$-$06 & 2.00 & 3.22e$-$06 & 3.28 & 1.58e$-$03 & 1.99 & 1.14e$-$03 & 3.04 \\\hline
32768 & 2.15e$-$09 & 3.38 & 2.06e$-$06 & 2.00 & 3.76e$-$07 & 3.10 & 3.96e$-$04 & 2.00 & 1.41e$-$04 & 3.01 \\\hline
		
	\end{tabular}
	\label{ss1:quadratic}
\end{table}

\begin{table}[htb!]
	\centering
	\caption{Discretisation errors  for the simply supported boundary condition: linear case and exact solution \eqref{ex.sol.2}}
	\begin{tabular}{|c|cc|cc|cc|cc|cc|}
		\hline
			\multirow{2}{*}{elem} & \multicolumn{2}{c|}{$\frac{\left\Vert u - u_h\right\Vert_{0,\Omega}}{\left\Vert u\right\Vert_{0,\Omega}}$} & \multicolumn{2}{c|}{$\frac{| u - u_h|_{1,\Omega}}{|u|_{1,\Omega}}$} & \multicolumn{2}{c|}{$\frac{\left\Vert \phi - \phi_h\right\Vert_{0,\Omega}}{\left\Vert \phi\right\Vert_{0,\Omega}}$} & \multicolumn{2}{c|}{$\frac{|\phi - \phi_h|_{1,\Omega}}{| \phi|_{1,\Omega}}$} & \multicolumn{2}{c|}{$\frac{\left\Vert \lambda - \lambda_h\right\Vert_{0,\Omega}}{\left\Vert \lambda\right\Vert_{0,\Omega}}$}\\ \hhline{~----------}
		& error & rate & error & rate & error & rate & error & rate & error & rate \\\hline
    8 & 1.77e+02 &   & 1.49e+02 &   & 8.84e+01 &          & 4.50e+01 &          & 1.95e+01 &          \\\hline
   32 & 6.18e+01 & 1.52 & 4.44e+01 & 1.75 & 2.21e+01 & 2.00 & 9.75 & 2.21 & 3.76 & 2.38 \\\hline
  128 & 1.98e+01 & 1.64 & 1.35e+01 & 1.72 & 6.40 & 1.79 & 2.74 & 1.83 & 1.13 & 1.74 \\\hline
  512 & 5.28 & 1.91 & 3.55 & 1.93 & 1.66 & 1.94 & 7.46e$-$01 & 1.88 & 3.16e$-$01 & 1.83 \\\hline
 2048 & 1.34 & 1.98 & 9.04e$-$01 & 1.98 & 4.20e$-$01 & 1.98 & 2.21e$-$01 & 1.76 & 8.21e$-$02 & 1.94 \\\hline
 8192 & 3.36e$-$01 & 1.99 & 2.30e$-$01 & 1.97 & 1.05e$-$01 & 2.00 & 7.94e$-$02 & 1.47 & 2.07e$-$02 & 1.98 \\\hline
32768 & 8.42e$-$02 & 2.00 & 6.12e$-$02 & 1.91 & 2.64e$-$02 & 2.00 & 3.48e$-$02 & 1.19 & 5.20e$-$03 & 2.00 \\\hline
131072 & 2.10e$-$02 & 2.00 & 1.85e$-$02 & 1.73 & 6.59e$-$03 & 2.00 & 1.67e$-$02 & 1.06 & 1.30e$-$03 & 2.00 \\\hline
524288 & 5.26e$-$03 & 2.00 & 6.94e$-$03 & 1.41 & 1.65e$-$03 & 2.00 & 8.27e$-$03 & 1.02 & 3.25e$-$04 & 2.00 \\\hline
	\end{tabular}
	\label{ss2:linear}
\end{table}

\begin{table}[htb!]
	\centering
	\caption{Discretisation errors  for the simply supported boundary condition: quadratic case and exact solution \eqref{ex.sol.2}}
	\begin{tabular}{|c|cc|cc|cc|cc|cc|}
		\hline
			\multirow{2}{*}{elem} & \multicolumn{2}{c|}{$\frac{\left\Vert u - u_h\right\Vert_{0,\Omega}}{\left\Vert u\right\Vert_{0,\Omega}}$} & \multicolumn{2}{c|}{$\frac{| u - u_h|_{1,\Omega}}{|u|_{1,\Omega}}$} & \multicolumn{2}{c|}{$\frac{\left\Vert \phi - \phi_h\right\Vert_{0,\Omega}}{\left\Vert \phi\right\Vert_{0,\Omega}}$} & \multicolumn{2}{c|}{$\frac{|\phi - \phi_h|_{1,\Omega}}{| \phi|_{1,\Omega}}$} & \multicolumn{2}{c|}{$\frac{\left\Vert \lambda - \lambda_h\right\Vert_{0,\Omega}}{\left\Vert \lambda\right\Vert_{0,\Omega}}$}\\ \hhline{~----------}
		& error & rate & error & rate & error & rate & error & rate & error & rate \\\hline
 8 & 3.07e$-$01 &   & 1.43 &   & 6.47 &          & 3.06e+01 &          & 1.53e+02 &          \\\hline
   32 & 4.05e$-$03 & 6.24 & 2.10e$-$02 & 6.09 & 1.33e$-$01 & 5.60 & 1.51 & 4.34 & 1.71e+01 & 3.16 \\\hline
  128 & 3.25e$-$04 & 3.64 & 2.44e$-$03 & 3.11 & 1.55e$-$02 & 3.11 & 3.27e$-$01 & 2.21 & 2.69 & 2.67 \\\hline
  512 & 2.20e$-$05 & 3.89 & 4.51e$-$04 & 2.44 & 1.29e$-$03 & 3.59 & 8.19e$-$02 & 2.00 & 3.06e$-$01 & 3.13 \\\hline
 2048 & 1.44e$-$06 & 3.93 & 1.10e$-$04 & 2.04 & 1.09e$-$04 & 3.56 & 2.11e$-$02 & 1.95 & 3.59e$-$02 & 3.09 \\\hline
 8192 & 1.01e$-$07 & 3.84 & 2.75e$-$05 & 2.00 & 1.10e$-$05 & 3.30 & 5.34e$-$03 & 1.98 & 4.39e$-$03 & 3.03 \\\hline
32768 & 8.32e$-$09 & 3.59 & 6.87e$-$06 & 2.00 & 1.28e$-$06 & 3.10 & 1.34e$-$03 & 2.00 & 5.46e$-$04 & 3.01 \\\hline
 
	\end{tabular}
	\label{ss2:quadratic}
\end{table}

From Tables \ref{ss1:linear} and \ref{ss2:linear}, we can see 
the quadratic convergence of errors in $L^2$-norm of the linear finite element method for 
$u$, $\phi$ and $\lambda$, whereas 
the convergence of errors in the $H^1$-seminorm for $u$ is slightly better than linear but for $\phi$ it is linear. 
We note that convergence in the $H^1$-seminorm for  $u$ is better in the earlier steps of refinement and as the 
refinement becomes finer and finer, the convergence rate becomes almost linear.

We have tabulated numerical results with the quadratic finite element method in 
Tables \ref{ss1:quadratic} and \ref{ss2:quadratic}. 
Working with the quadratic finite element we see slightly better than $O(h^{3})$ rate of 
convergence for the convergence of the errors in $L^2$-norm for $u$, whereas 
the convergence is of $O(h^{2})$ for the errors in the semi $H^1$-norm. Similarly, 
the errors in the $L^2$-norm for $\phi$ and $\lambda$ converge with order 
$O(h^3)$, respectively, whereas the errors in the semi $H^1$-norm of $\phi$ converge with $O(h^2)$. 
The numerical results follow the predicted theoretical rates also for both examples.

\paragraph{Example 3} 
In the third example we consider the exact solution satisfying 
$u=0,\Delta u=0, \Delta^2 u = 0$ but $\nabla u\cdot \mathbf{n}\neq 0$ on the boundary:
\begin{equation} \label{exa3}
u = \sin\left(\pi x\right) \sin\left(\pi y\right) \text{ in } \Omega =\left( 0,1\right)^{2}.
\end{equation}
We note that the exact solutions chosen for Examples 1 and 2 satisfy 
 $\nabla u\cdot \mathbf{n} = 0$ on the boundary of the domain $\Omega$.

We start with the initial mesh as given in the left picture of Figure \ref{imesh}
and compute the relative errors in various norms for all  three variables at each step of refinement.
The computed errors in different norms are tabulated in Tables \ref{ss3:linear} and 
\ref{ss3:quadratic}. 
Interestingly, we still get the same rate of convergence for most of the norms with two exceptions: (i) in case of the linear finite element method, we do not observe a super-convergence rate in $H^1$-norm of $u$, and (ii) in the quadratic finite element method, the rate of convergence in $L^2$-norm of $u$ is only $O(h^3)$.

\begin{table}[htb!]
	\centering
	\caption{Discretisation errors  for the simply supported boundary condition: linear case and exact solution \eqref{exa3}}
	\begin{tabular}{|c|cc|cc|cc|cc|cc|}
		\hline
				\multirow{2}{*}{elem} & \multicolumn{2}{c|}{$\frac{\left\Vert u - u_h\right\Vert_{0,\Omega}}{\left\Vert u\right\Vert_{0,\Omega}}$} & \multicolumn{2}{c|}{$\frac{| u - u_h|_{1,\Omega}}{|u|_{1,\Omega}}$} & \multicolumn{2}{c|}{$\frac{\left\Vert \phi - \phi_h\right\Vert_{0,\Omega}}{\left\Vert \phi\right\Vert_{0,\Omega}}$} & \multicolumn{2}{c|}{$\frac{|\phi - \phi_h|_{1,\Omega}}{| \phi|_{1,\Omega}}$} & \multicolumn{2}{c|}{$\frac{\left\Vert \lambda - \lambda_h\right\Vert_{0,\Omega}}{\left\Vert \lambda\right\Vert_{0,\Omega}}$}\\ \hhline{~----------}
		& error & rate & error & rate & error & rate & error & rate & error & rate \\\hline
    8 & 8.42e$-$01 &  & 8.78e$-$01 &  & 7.47e$-$01 &  & 1.23e+03 &  & 5.95e$-$01 &  \\\hline
    32 & 4.22e$-$01 & 1.00 & 4.91e$-$01 & 0.84 & 3.30e$-$01 & 1.18 & 6.65e+02 & 0.89 & 2.27e$-$01 & 1.39 \\\hline
    128 & 1.32e$-$01 & 1.68 & 2.18e$-$01 & 1.17 & 9.83e$-$02 & 1.75 & 3.13e+02 & 1.08 & 6.42e$-$02 & 1.83 \\\hline
    512 & 3.50e$-$02 & 1.91 & 1.01e$-$01 & 1.11 & 2.57e$-$02 & 1.93 & 1.52e+02 & 1.04 & 1.65e$-$02 & 1.95 \\\hline
    2048 & 8.88e$-$03 & 1.98 & 4.95e$-$02 & 1.03 & 6.50e$-$03 & 1.98 & 7.53e+01 & 1.01 & 4.16e$-$03 & 1.99 \\\hline
    8192 & 2.22e$-$03 & 1.99 & 2.46e$-$02 & 1.01 & 1.63e$-$03 & 2.00 & 3.76e+01 & 1.00 & 1.04e$-$03 & 2.00 \\\hline
    32768 & 5.58e$-$04 & 2.00 & 1.23e$-$02 & 1.00 & 4.08e$-$04 & 2.00 & 1.87e+01 & 1.00 & 2.61e$-$04 & 2.00 \\\hline
	\end{tabular}
	\label{ss3:linear}
\end{table}

\begin{table}[htb!]
	\centering
	\caption{Discretisation errors  for the simply supported boundary condition: quadratic case and exact solution \eqref{exa3}}
	\begin{tabular}{|c|cc|cc|cc|cc|cc|}
		\hline
		\multirow{2}{*}{elem} & \multicolumn{2}{c|}{$\frac{\left\Vert u - u_h\right\Vert_{0,\Omega}}{\left\Vert u\right\Vert_{0,\Omega}}$} & \multicolumn{2}{c|}{$\frac{| u - u_h|_{1,\Omega}}{|u|_{1,\Omega}}$} & \multicolumn{2}{c|}{$\frac{\left\Vert \phi - \phi_h\right\Vert_{0,\Omega}}{\left\Vert \phi\right\Vert_{0,\Omega}}$} & \multicolumn{2}{c|}{$\frac{|\phi - \phi_h|_{1,\Omega}}{| \phi|_{1,\Omega}}$} & \multicolumn{2}{c|}{$\frac{\left\Vert \lambda - \lambda_h\right\Vert_{0,\Omega}}{\left\Vert \lambda\right\Vert_{0,\Omega}}$}\\ \hhline{~----------}
		& error & rate & error & rate & error & rate & error & rate & error & rate \\\hline
    8 & 2.12e$-$01 &  & 2.29e$-$01 &  & 1.63e$-$01 &  & 2.90e+02 &  & 1.14e$-$01 &  \\\hline
    32 & 2.12e$-$02 & 3.33 & 4.22e$-$02 & 2.44 & 1.70e$-$02 & 3.26 & 6.20e+01 & 2.22 & 1.36e$-$02 & 3.08 \\\hline
    128 & 1.98e$-$03 & 3.42 & 9.99e$-$03 & 2.08 & 1.78e$-$03 & 3.26 & 1.52e+01 & 2.03 & 1.63e$-$03 & 3.05 \\\hline
    512 & 2.14e$-$04 & 3.21 & 2.49e$-$03 & 2.01 & 2.07e$-$04 & 3.11 & 3.80 & 2.00 & 2.01e$-$04 & 3.02 \\\hline
    2048 & 2.54e$-$05 & 3.07 & 6.21e$-$04 & 2.00 & 2.52e$-$05 & 3.03 & 9.51e$-$01 & 2.00 & 2.51e$-$05 & 3.00 \\\hline
    8192 & 3.14e$-$06 & 3.02 & 1.56e$-$04 & 2.00 & 3.14e$-$06 & 3.01 & 2.38e$-$01 & 2.00 & 3.14e$-$06 & 3.00 \\\hline	
	\end{tabular}
	\label{ss3:quadratic}
\end{table}

\subsection{Clamped boundary conditions}

\paragraph{Example 1}
We choose the exact solution 
\begin{equation}\label{ex.sol.3}
u = 4096x^3(1-x)^3y^3(1-y)^3 \text{ in $\Omega = \left(0,1\right)^2$}, 
\end{equation}
so that the exact solution satisfies the clamped boundary condition 
\[ u = \Delta u = \frac{\partial u}{\partial \textbf{n}} =0 \quad\text{on}\quad 
\partial \Omega.\]
For our clamped boundary condition 
we start with the initial mesh as given in the right picture of Figure \ref{imesh}. 
In the following $\phi$ and $\lambda$ are discretised using the linear finite element space, whereas 
$u$ is discretised using the quadratic finite element space. That means we use the finite element spaces 
with $k=1$. 
The numerical results are tabulated in  Table \ref{clbc1}. In this example, we get 
higher convergence rates than predicted by the theory for all errors. 
These results seem to indicate a higher order of convergence for $u$ in the semi
 $H^1$-norm than in 
the $L^2$-norm. This can
either be due to the asymptotic rates not being achieved at the grid levels considered, or to some genuine
super-convergence result. Understanding this phenomenon in more depth is the purpose of future work.
As in the case of Examples 1 and 2 of the simply supported boundary condition we see the 
better convergence rates for the  semi $H^1$-norm in earlier steps of refinement. 
However, when we refine further the convergence rates decrease close to 2. Thus the better convergence rates 
are due to the asymptotic rates not being achieved at the earlier steps of refinement.

\begin{table}[htb!]
	\centering
	\caption{Discretisation errors  for the clamped boundary condition: exact solution \eqref{ex.sol.3}}
	\begin{tabular}{|c|cc|cc|cc|cc|cc|}
		\hline
			\multirow{2}{*}{elem} & \multicolumn{2}{c|}{$\frac{\left\Vert u - u_h\right\Vert_{0,\Omega}}{\left\Vert u\right\Vert_{0,\Omega}}$} & \multicolumn{2}{c|}{$\frac{| u - u_h|_{1,\Omega}}{|u|_{1,\Omega}}$} & \multicolumn{2}{c|}{$\frac{\left\Vert \phi - \phi_h\right\Vert_{0,\Omega}}{\left\Vert \phi\right\Vert_{0,\Omega}}$} & \multicolumn{2}{c|}{$\frac{|\phi - \phi_h|_{1,\Omega}}{| \phi|_{1,\Omega}}$} & \multicolumn{2}{c|}{$\frac{\left\Vert \lambda - \lambda_h\right\Vert_{0,\Omega}}{\left\Vert \lambda\right\Vert_{0,\Omega}}$}\\ \hhline{~----------}
		& error & rate & error & rate & error & rate & error & rate & error & rate \\\hline
 32 & 4.34 &   & 8.51 &   & 7.47e$-$01 &          & 9.69e$-$01 &          & 9.20e$-$01 &          \\\hline
  128 & 1.09 & 2.00 & 3.46 & 1.30 & 3.06e$-$01 & 1.29 & 5.38e$-$01 & 0.85 & 3.88e$-$01 & 1.25 \\\hline
  512 & 1.85e$-$01 & 2.55 & 6.43e$-$01 & 2.43 & 1.26e$-$01 & 1.29 & 2.86e$-$01 & 0.91 & 2.13e$-$01 & 0.86 \\\hline
 2048 & 2.43e$-$02 & 2.93 & 7.73e$-$02 & 3.06 & 2.38e$-$02 & 2.40 & 1.30e$-$01 & 1.14 & 9.34e$-$02 & 1.19 \\\hline
 8192 & 4.76e$-$03 & 2.35 & 9.94e$-$03 & 2.96 & 4.49e$-$03 & 2.40 & 6.35e$-$02 & 1.03 & 2.40e$-$02 & 1.96 \\\hline
32768 & 1.11e$-$03 & 2.11 & 1.39e$-$03 & 2.84 & 1.02e$-$03 & 2.13 & 3.17e$-$02 & 1.00 & 4.67e$-$03 & 2.36 \\\hline
131072 & 2.73e$-$04 & 2.02 & 2.74e$-$04 & 2.34 & 2.53e$-$04 & 2.02 & 1.58e$-$02 & 1.00 & 8.24e$-$04 & 2.50 \\\hline
	\end{tabular}
	\label{clbc1}
\end{table}
\paragraph{Example 2}
For our last example with clamped boundary condition the exact solution is 
chosen as 
\begin{equation}\label{ex.sol.4}
u = 4096x^3(1-x)^3y^3(1-y)^3\left(\frac{2}{5}e^x + \cos(y)\right).
\end{equation}
As in the previous example, this solution also satisfies the clamped boundary condition. 
We have tabulated the relative error in various norms in Table \ref{clbc2}. 
The results are very similar to the ones as in the first example. However, the relative error in the case of 
clamped boundary conditions are higher than in the case of simply supported boundary conditions. 
We can also see that the asymptotic rates of error reduction start later in this case 
due to the extrapolation on the boundary patch of the domain. 
\begin{table}[htb!]
	\centering
	\caption{Discretisation errors  for the clamped boundary condition: exact solution \eqref{ex.sol.4}}
	\begin{tabular}{|c|cc|cc|cc|cc|cc|}
		\hline
			\multirow{2}{*}{elem} & \multicolumn{2}{c|}{$\frac{\left\Vert u - u_h\right\Vert_{0,\Omega}}{\left\Vert u\right\Vert_{0,\Omega}}$} & \multicolumn{2}{c|}{$\frac{| u - u_h|_{1,\Omega}}{|u|_{1,\Omega}}$} & \multicolumn{2}{c|}{$\frac{\left\Vert \phi - \phi_h\right\Vert_{0,\Omega}}{\left\Vert \phi\right\Vert_{0,\Omega}}$} & \multicolumn{2}{c|}{$\frac{|\phi - \phi_h|_{1,\Omega}}{| \phi|_{1,\Omega}}$} & \multicolumn{2}{c|}{$\frac{\left\Vert \lambda - \lambda_h\right\Vert_{0,\Omega}}{\left\Vert \lambda\right\Vert_{0,\Omega}}$}\\ \hhline{~----------}
		& error & rate & error & rate & error & rate & error & rate & error & rate \\\hline
 32 & 8.38 &   & 1.33e+01 &   & 7.77e$-$01 &          & 1.00 &          & 1.19 &          \\\hline
  128 & 1.36 & 2.63 & 4.01 & 1.73 & 4.45e$-$01 & 0.80 & 6.54e$-$01 & 0.62 & 6.40e$-$01 & 0.90 \\\hline
  512 & 2.05e$-$01 & 2.72 & 7.61e$-$01 & 2.40 & 1.30e$-$01 & 1.78 & 2.89e$-$01 & 1.18 & 2.20e$-$01 & 1.54 \\\hline
 2048 & 2.46e$-$02 & 3.06 & 8.75e$-$02 & 3.12 & 2.40e$-$02 & 2.44 & 1.30e$-$01 & 1.15 & 9.39e$-$02 & 1.23 \\\hline
 8192 & 4.76e$-$03 & 2.37 & 1.07e$-$02 & 3.03 & 4.53e$-$03 & 2.40 & 6.36e$-$02 & 1.03 & 2.41e$-$02 & 1.96 \\\hline
32768 & 1.10e$-$03 & 2.11 & 1.45e$-$03 & 2.89 & 1.03e$-$03 & 2.13 & 3.17e$-$02 & 1.00 & 4.71e$-$03 & 2.36 \\\hline
131072 & 2.72e$-$04 & 2.02 & 2.78e$-$04 & 2.38 & 2.55e$-$04 & 2.02 & 1.59e$-$02 & 1.00 & 8.35e$-$04 & 2.50 \\\hline
	\end{tabular}
	\label{clbc2}
\end{table}

\begin{remark}
We have proved the error estimate in the mesh-dependent norm $|\cdot |_{k,h}$ for the clamped boundary condition case. 
This norm can be estimated by the standard $L^2$-norm and $H^1$-norm as follows. 
Using the triangle inequality we have 
\begin{eqnarray*}
\snorm{(u-u_h,\phi-\phi_h)}{k,h}^2
&=& \|\phi -\phi_h - \Delta_h u + \Delta_h u_h\|^2_{0,\Omega} + \|\nabla (\phi- \phi_h)\|^2_{0,\Omega} \\
&\leq& C(  \|\phi -\phi_h \|^2_{0,\Omega} +\| \Delta_h u - \Delta_h u_h\|^2_{0,\Omega} +  \|\nabla (\phi- \phi_h)\|^2_{0,\Omega}).
\end{eqnarray*}
We now only consider the middle term of the last line  of the last inequality. 
Using  the definition of $\Delta_h$,  the  $L^2$-norm  and the standard inverse estimate we have a constant $C$ independent of 
the mesh-size $h$ such that 
\begin{eqnarray*} \| \Delta_h u - \Delta_h u_h\|_{0,\Omega} 
&\leq & 
\sup_{\phi_h \in S_{h,0}^{k+1}} \frac{\int_{\Omega} (\nabla u - \nabla u_h) \cdot \nabla \phi_h\,dx}{\|\phi_h\|_{0,\Omega}}\\
&\leq & \frac{C}{h} \|\nabla u - \nabla u_h\|_{0,\Omega}.
\end{eqnarray*}
Since the computed errors behave like   $\|\phi -\phi_h \|_{0,\Omega}=O(h^2) $,  $ \|\nabla u - \nabla u_h\|_{0,\Omega}=O(h^2)$ and 
 $\|\nabla \phi - \nabla \phi_h\|_{0,\Omega}=O(h)$, the errors for $u$ and $\phi$ in the mesh-dependent norm $|\cdot|_{k,h}$ behave as
   $\snorm{(u-u_h,\phi-\phi_h)}{k,h}=O(h)$.
\end{remark}

\section*{Acknowledgements}
\begin{itemize}
\item We are  grateful to the anonymous referees for their valuable suggestions 
to improve the quality of the earlier version of this work.

\item Part of this work was completed during a visit of the fourth author to the University of Newcastle. He is grateful for their hospitality.
The fourth author was partially supported by an Australian Research Council (ARC) grant DP120100097.
He is currently partially supported by ARC grant DP150100375.

\item The first author is partially supported by ARC grant DP170100605.
\end{itemize}

\appendix

\section{Proof of Theorem \ref{TMexistence}}

The existence and uniqueness of a solution
to \eqref{wbiharm} follows from the Ladyzenskaia--Babushka--Brezzi theory, provided that
we establish the following properties.
\begin{enumerate}
\item The bilinear forms $a(\cdot,\cdot)$, $b(\cdot,\cdot)$ and the linear form 
$\ell(\cdot)$ are continuous on $V \times V$, $V \times \M$ and 
$V$, respectively. 
\item The bilinear form $a(\cdot,\cdot)$ is coercive on the kernel space 
\[ \CV =\{ (v, \psi) \in V :\, b((v,\psi),\mu) = 0,\; \mu\in \M\}.\]  
\item The bilinear form $b(\cdot,\cdot)$ satisfies the inf--sup condition, for some $\beta>0$:
\[ \inf_{\mu \in \M} \sup_{(v, \psi)  \in V} \frac{b((v,\psi),\mu)} 
{ \|(v,\psi)\|_{V} \|\mu\|_{\M}} \geq \beta. \]

\end{enumerate}
The Cauchy--Schwarz inequality implies that the bilinear forms $a(\cdot,\cdot)$, $b(\cdot,\cdot)$ and the linear form 
$\ell(\cdot)$ are continuous on $V \times V$, $V \times \M$ and 
$V$, respectively. 
 We now turn our attention to the second condition. 
In fact, for $(u, \phi) \in V $ satisfying $ b((u,\phi),\mu) = 0$ for all $\mu \in \M\supset H^1_0(\Omega)$ we have  with $\mu=u$ 
\[ \int_{\Omega} \nabla u \cdot \nabla u \,dx = - \int_{\Omega} \phi u\,dx.\]
Hence, using Cauchy--Schwarz and the Poincar\'e inequality we find 
\[ \|\nabla u \|^2_{0,\Omega} \leq C \|\phi\|_{0,\Omega} \|\nabla u\|_{0,\Omega}.\]
Thus we have 
\[ \|\nabla u \|_{0,\Omega}  \leq C \|\phi\|_{0,\Omega}.\]
From this inequality we  infer 
\[ \|\nabla u \|^2_{0,\Omega}+  \|\nabla \phi\|^2_{0,\Omega}
\leq C \|\phi\|^2_{0,\Omega} + \|\nabla \phi\|^2_{0,\Omega}.\]
We use the Poincar\'e inequality again to obtain the coercivity 
\[ \|\nabla u \|^2_{0,\Omega}+  \|\nabla \phi\|^2_{0,\Omega}\leq C
 \|\phi\|^2_{0,\Omega} + \|\nabla \phi\|^2_{0,\Omega}
\leq C a((u,\phi),(u,\phi)), \quad (u,\phi) \in \CV.\]

Let us now consider the inf--sup condition in the case of simply supported BCs, that
is $\M=H^1_0(\Omega)$ with natural norm.
For all $\mu\in H^1_0(\Omega)$,
\[
b((\mu,0),\mu)=-\langle \mu,\Delta\mu\rangle = \int_\Omega |\nabla\mu|^2=\|\mu\|_{H^1_0(\Omega)}^2
\]
and thus 
\[
\sup_{(v, \psi)  \in V} \frac{b((v,\psi),\mu)} 
{ \|(v,\psi)\|_{V} }\ge \frac{b((\mu,0),\mu)}{\|\mu\|_{H^1_0(\Omega)}}\ge \|\mu\|_{H^1_0(\Omega)}.
\]

We finally consider the inf--sup condition in the case of clamped boundary conditions,
for which $\M=\{\mu\in H^{-1}(\Omega)\,:\,\Delta \mu\in H^{-1}(\Omega)\}$ with corresponding
graph norm.
We have
\[  \sup_{(v, \psi)  \in V} \frac{b((v,\psi),\mu)} 
{ \|(v,\psi)\|_{V} }  = \sup_{(v, \psi)  \in V} 
\frac{ {\langle\psi,\mu\rangle}-\langle v, \Delta \mu \rangle } 
{ \|(v,\psi)\|_V }\,.   \]
Now setting $\psi=0$ we obtain 
\[ \sup_{(v, \psi)  \in V} \frac{b((v,\psi),\mu)} 
{ \|(v,\psi)\|_{V} } 
\geq \sup_{ v  \in {H^1_0(\Omega)}} 
\frac{ \langle v,\Delta\mu\rangle} 
{ \|\nabla v\|_{0,\Omega} }   \geq c_1 \|\Delta \mu\|_{-1,\Omega},\]
where we have used Poincar\'e inequality in the last step. 
Similarly, using $v=0$ we get 
\[ \sup_{(v, \psi)  \in V} \frac{b((v,\psi),\mu)} 
{ \|(v,\psi)\|_{V} } 
\geq \sup_{ \psi  \in H^1_0(\Omega)} 
\frac{\langle\psi,\mu\rangle}  
{ \|\nabla \psi\|_{0,\Omega} }   \geq  c_2 \|\mu\|_{-1,\Omega},\]
and, hence, there exists a constant $\beta>0$ such that 
\[  \sup_{(v, \psi)  \in V} \frac{b((v,\psi),\mu)} 
{ \|(v,\psi)\|_{V} }  =\sup_{(v, \psi)  \in V} 
\frac{ \langle \psi\,\mu\rangle-\langle v, \Delta \mu \rangle } 
{ \|(v,\psi)\|_{V} }\geq \beta \|\mu\|_\M.   \]
Hence,  \eqref{wbiharm} has a unique solution. 
{{\quad \ \vbox{\hrule\hbox{%
   \vrule height1.0ex\hskip1.0ex\vrule}\hrule
  }}\par}

\bibliographystyle{siam}
\bibliography{total}
\end{document}